\documentclass[12pt]{amsart}
\usepackage{amscd,amsmath,amsthm,amssymb,graphics}
\usepackage{pstcol,pst-plot,pst-3d}%\usepackage[T1]{fontenc}
\usepackage{lmodern,pst-node}%\usepackage{geometric}
\usepackage{multicol}
\usepackage{epic,eepic}
\usepackage{amsfonts,amssymb,amscd,amsmath,enumitem,verbatim}
\psset{unit=0.7cm,linewidth=0.8pt,arrowsize=2.5pt 4}
\newpsstyle{fatline}{linewidth=1.5pt}
\newpsstyle{fyp}{fillstyle=solid,fillcolor=verylight}
\definecolor{verylight}{gray}{0.97}
\definecolor{light}{gray}{0.9}
\definecolor{medium}{gray}{0.85}
\definecolor{dark}{gray}{0.6}

\def\NZQ{\Bbb}               % the font for N,Z,Q,R,C

\def\FF{{\NZQ F}}
\def\GG{{\NZQ G}}

\def\KK{{\NZQ K}}
\def\frk{\frak}               % font for "Fraktur"

\def\Phi{{\frk n}}
\def\Phi{{\frk N}}
\def\MF{{\mathcal F}}

\def\MS{{\mathcal S}}

\def\MA{{\mathcal A}}
\def\MN{{\mathcal N}}

\def\opn#1#2{\def#1{\operatorname{#2}}} % to make operators
\opn\chara{char} \opn\length{\ell} \opn\pd{pd} \opn\rk{rk}
\opn\projdim{proj\,dim} \opn\injdim{inj\,dim} \opn\rank{rank}
\opn\depth{depth} \opn\grade{grade} \opn\height{height}
\opn\embdim{emb\,dim} \opn\codim{codim}

\opn\Tr{Tr} \opn\bigrank{big\,rank}
\opn\superheight{superheight}\opn\lcm{lcm}
\opn\trdeg{tr\,deg}%\emph{
\opn\reg{reg} \opn\lreg{lreg} \opn\ini{in} \opn\lpd{lpd}
\opn\size{size}\opn\bigsize{bigsize}
\opn\cosize{cosize}\opn\bigcosize{bigcosize}
\opn\sdepth{sdepth}\opn\sreg{sreg}
\opn\link{link}\opn\fdepth{fdepth}
\opn\div{div} \opn\Div{Div} \opn\cl{cl} \opn\Cl{Cl}
\opn\Spec{Spec} \opn\Supp{Supp} \opn\supp{supp} \opn\Sing{Sing}
\opn\Ass{Ass} \opn\Min{Min}\opn\Mon{Mon} \opn\dstab{dstab} \opn\astab{astab}
\opn\Syz{Syz}
\opn\Ann{Ann} \opn\Rad{Rad} \opn\Soc{Soc}
\opn\Im{Im} \opn\Ker{Ker} \opn\Coker{Coker} \opn\Am{Am}
\opn\Hom{Hom} \opn\Tor{Tor} \opn\Ext{Ext} \opn\End{End}
\opn\Aut{Aut} \opn\id{id}

\opn\nat{nat}
\opn\pff{pf}%   \pf exists already
\opn\Pf{Pf} \opn\GL{GL} \opn\SL{SL} \opn\mod{mod} \opn\ord{ord}
\opn\Gin{Gin} \opn\Hilb{Hilb}\opn\sort{sort}
\opn\initial{init}
\opn\ende{end}
\opn\height{height}
\opn\type{type}
\opn\aff{aff} \opn\con{conv} \opn\relint{relint} \opn\st{st}
\opn\lk{lk} \opn\cn{cn} \opn\core{core} \opn\vol{vol}
\opn\link{link} \opn\star{star}\opn\lex{lex}
\opn\gr{gr}

\def\pot#1#2{#1[\kern-0.28ex[#2]\kern-0.28ex]}

\opn\dirlim{\underrightarrow{\lim}}
\opn\inivlim{\underleftarrow{\lim}}
\let\union=\cup
\let\sect=\cap

\let\tensor=\otimes
\let\iso=\cong

\let\Sect=\bigcap

\let\to=\rightarrow
\let\To=\longrightarrow
\def\Implies{\ifmmode\Longrightarrow \else
        \unskip${}\Longrightarrow{}$\ignorespaces\fi}
\def\implies{\ifmmode\Rightarrow \else
        \unskip${}\Rightarrow{}$\ignorespaces\fi}
\def\iff{\ifmmode\Longleftrightarrow \else
        \unskip${}\Longleftrightarrow{}$\ignorespaces\fi}

\let\:=\colon
 \theoremstyle{plain}
\newtheorem{Theorem}{Theorem}[section]
 
 \newtheorem{Corollary}[Theorem]{Corollary}

 \theoremstyle{definition}

\let\epsilon\varepsilon
\let\kappa=\varkappa
\def\qed{\ifhmode\textqed\fi
      \ifmmode\ifinner\quad\qedsymbol\else\dispqed\fi\fi}
\def\textqed{\unskip\nobreak\penalty50
       \hskip2em\hbox{}\nobreak\hfil\qedsymbol
       \parfillskip=0pt \finalhyphendemerits=0}
\def\dispqed{\rlap{\qquad\qedsymbol}}

\opn\dis{dis}
\def\pnt{{\raise0.5mm\hbox{\large\bf.}}}

\opn\Lex{Lex}
\begin{document}
\title{The face ideal of a simplicial complex}
\author {J\"urgen Herzog and  Takayuki Hibi}

\address{J\"urgen Herzog, Fachbereich Mathematik, Universit\"at Duisburg-Essen, Campus Essen, 45117
Essen, Germany} \email{juergen.herzog@uni-essen.de}

\address{Takayuki Hibi, Department of Pure and Applied Mathematics, Graduate School of Information Science and Technology,
Osaka University, Toyonaka, Osaka 560-0043, Japan}
\email{hibi@math.sci.osaka-u.ac.jp}

\begin{abstract}
Given a simplicial complex we associate to it a squarefree monomial ideal which we call the face ideal of the simplicial complex, and show that it has linear quotients. It turns out that its Alexander dual is a whisker complex. We apply this construction in particular to chain and antichain ideals of a finite partially ordered set. We also introduce so-called higher dimensional whisker complexes and show that their independence  complexes are shellable.
\end{abstract}
\thanks{}
\subjclass[2010]{}

\maketitle
\section*{Introduction}
Let $S=K[x_1,\ldots,x_n,y_1,\ldots,y_n]$ denote the polynomial ring
in $2n$ variables over the field $K$.  In general, given any subset $F$ of
$\{ x_{1}, \ldots, x_{n} \}$,
we define the squarefree monomial $u_{F}$ of $S$ of degree $n$ by setting
\[
u_F = \prod_{x_{i} \in F} x_{i}
\prod_{x_{j} \in \{ x_{1}, \ldots, x_{n} \} \setminus F} y_{j}.
\]
Given a collection $\MS$ of subsets of $\{ x_{1}, \ldots, x_{n} \}$ one defines the ideal $I_\MS$ generated by the monomials $u_F$ with $F\in \MS$. Of particular interest are collections of sets which naturally arise in combinatorics.

The first example of this kind which appeared in the literature is the following: let $P = \{ x_{1}, \ldots, x_{n} \}$ be a finite partially ordered set.
A {\em poset ideal} of $P$ is a subset $\alpha$  of $P$ with the property that   if
$x_{i} \in \alpha$ and $x_{j} \leq x_{i}$, then $x_{j} \in \alpha$.
In particular the empty set as well as $P$ itself is a poset ideal of $P$.
The squarefree monomial ideal $I(P) \subset S$ which is generated by those monomials
$u_{\alpha}$ for which $\alpha$ is a poset ideal of $P$ has played an
important role in combinatorial and computational commutative algebra
(\cite{HH1}, \cite{EHM}).  The ideal  $I(P)$ has the remarkable property that it has a linear reolution \cite[Theorem 9.1.8]{HH}.

Now let   $I_{A}(P) \subset S$ be the squarefree monomial ideal
which is generated by those monomials $u_{\beta}$ for which $\beta$ is an antichain
of $P$.  (Recall that an {\em antichain} of $P$ is a subset $\beta \subset P$
for which any two elements $x_{i}$ and $x_{j}$ with $i \neq j$
belonging to $\beta$ are incomparable in $P$.) The toric ring generated by $u_\alpha$ with $u_\alpha\in I(P)$ and the  toric ring generated by $u_\beta$ with $u_\beta\in I_A(P)$ have similar properties. In particular, both are algebras with straighening laws on suitable distributive lattices, see \cite{H} and \cite{HL}. Thus one may  expect that  $I(P)$ and $I_A(P)$ have similar properties as well.

Similarly we may also consider
the squarefree monomial ideal $I_{C}(P) \subset S$
which is generated by those monomials  $u_{\gamma}$ for which $\gamma$ is a chain of $P$.
(Recall that a {\em chain} of $P$ is a totally ordered subset of $P$.)
As expected, each of the antichain ideal $I_{A}(P)$ and
the chain ideal $I_{C}(P)$ has linear quotients will be shown in
Theorem \ref{chain}.

However, far beyond the study of
antichain ideals and chain ideals of partially ordered sets,
the study of the present paper will be done in a much more general situation.
Since the set of antichains of $P$ as well as the set of chains of $P$
is clearly a simplicial complex on the vertex set $\{ x_{1}, \ldots, x_{n} \}$,
it is natural to study more generally ideals $I_\MS$ where $\MS$ is the set of faces of any simplicial complex. Thus we introduce the face ideal $J_{\Delta}$
of a simplicial $\Delta$ on $\{ x_{1}, \ldots, x_{n} \}$.
In other words, the face ideal $J_{\Delta}$ is the ideal
of $S$ which is generated by the monomials $u_F$ with $F\in \Delta$.  Theorem \ref{faceideal} gives
the structure of the Alexander dual $(J_{\Delta})^{\vee}$ of $J_{\Delta}$.
Somewhat surprisingly, it turns out that  $(J_{\Delta})^{\vee}$ is
a whisker complex, which is a generalization of whisker graphs,  first introduced by Villarreal \cite{V} and further  studied and generalized in \cite{CN}, \cite{FH}, \cite{HHKO} and \cite{VTV}.
A simple polarization argument shows that  $(J_{\Delta})^{\vee}$ is Cohen--Macaulay.  Hence, again.
$J_{\Delta}$ has a linear resolution.
We also describe the explicit minimal free resolution of $J_{\Delta}$
(Theorem \ref{resolution}) and compute the Betti numbers of $J_{\Delta}$
(Corollary \ref{betti}). We would like to mention that whisker complexes are special classes of grafted complexes as introduced by  Faridi \cite{F}.

In Section $2$, we show that the face ideal $J_{\Delta}$ of a simplicial
complex $\Delta$ has linear quotients.  This fact implies  that  the independence complex  of the whisker complex of an arbitrary simplicial complex
is shellable (Corollary \ref{shellable}). Recall that if $\Gamma$ is a simplicial complex and $I(\Gamma)$ its facet ideal, then the simplicial complex $\Delta$ with $I_\Delta=I(\Gamma)$ is called the {\em independence complex} of $\Gamma$. Here $I_\Delta$ denotes the Stanley--Reisner ideal of $\Delta$.  The  faces of $\Delta$ are those subsets of the vertex set of $\Gamma$ which do not contain any facet of $\Gamma$.

Dochtermann and Engstr\"om \cite{DE} even showed that the independence complex of  whisker graphs are pure and vertex decomposable, which in particular implies that the independence complex of any whisker graph is shellable, see also \cite{HHKO}.

Finally, in Section $4$, we introduce the concept of
higher dimensional whisker complexes, which is a generalization of
whisker complexes introduced in Section $1$.
Theorem~\ref{generalized} guarantees that hte independence  complex
of the higher dimensional whisker complex of
an arbitrary simplicial complex
is shellable. Thus our Theorem~\ref{generalized} generalizes the result of Dochtermann and Engstr\"om regarding shellability.

\medskip
A.~Engstr\"om  kindly informed us that his student Lauri Loiskekoski in his Master Thesis \cite{LL} has also introduced what we call the face ideal of a simplicial complex, and, among other results, has  shown that face ideals have linear resolutions, cf.\ our Corollary~\ref{always}.

\section{Face ideals and whisker complexes}
Let $S=K[x_1,\ldots,x_n,y_1,\ldots,y_n]$ denote the polynomial ring in $2n$ variables over the field $K$ and $\Delta$ a simplicial complex on the vertex set $\{x_1,\ldots,x_n\}$.

The \emph{whisker complex}  of $\Delta$
is the simplicial complex
$W(\Delta)$ on the vertex set $\{x_1,\ldots,x_n,y_1,\ldots,y_n\}$
which is obtained from $\Delta$ by adding the facets $\{x_i,y_i\}$, called whiskers,
for $i=1,\ldots,n$.

For each face $F\in \Delta$ we associate the monomial $u_F\in S$ defined by
\[
u_F=x_Fy_{F^c},
\]
where
\[
x_F=\prod_{x_i\in F}x_i\quad \text{and} \quad y_{F^c}=\prod_{x_j\in \{x_1,\ldots,x_n\} \setminus F}y_j.
\]
The ideal of $S$ generate by those squarefree monomials $u_F$ with $F \in \Delta$
is called
the \emph{face ideal} of $\Delta$ and is denoted by $J_\Delta$.
As usual we write $I_\Delta$ $(\subset K[x_{1}, \ldots, x_{n}])$
for the Stanley--Reisner ideal (\cite[p.~16]{HH}) of $\Delta$ and
$I(\Delta)$ $(\subset K[x_{1}, \ldots, x_{n}])$ for the facet ideal of $\Delta$.

Let, in general, $I\subset S$ be a squarefree monomial ideal of $S$ with
$I=\Sect_{k=1}^mP_k$ where the ideals $P_k$ are the minimal prime ideals of $I$.
Each $P_k$ is generated by variables.
The \emph{Alexander dual} $I^{\vee}$ of $I$ is defined to be
the ideal of $S$ generated by the squarefree monomials $u_1,\ldots,u_m$, where
$u_k=(\prod_{x_i\in P_k}x_i)(\prod_{y_j\in P_k}y_j)$. In particular, $(I_\Delta)^\vee=I_{\Delta^\vee}$ where $\Delta^\vee$ is the Alexander dual of $\Delta$.

\begin{Theorem}
\label{faceideal}
Let $\Delta$ be a simplicial complex on
$\{x_{1}, \ldots, x_{n}\}$ and $J_{\Delta} \subset S$ the face ideal of $\Delta$.
Then one has
\[
(J_\Delta)^\vee=I(W(\Gamma)),
\]
where $\Gamma$ is the simplicial complex on $\{ y_{1}, \ldots, y_{n} \}$ with
\[
I_{\Delta'}=I(\Gamma),
\]
where
$\Delta'$ is the copy of $\Delta$ on $\{y_{1}, \ldots, y_{n}\}$, i.e.,
$\Delta' = \{ \, \{ \, y_{i} \, : \, x_{i} \in F \, \} \, : \, F \in \Delta \, \}$.
\end{Theorem}

\begin{proof}
Let $\Delta^{\sharp}$ denote the simplicial complex on
$\{ x_{1}, \ldots, x_{n}, y_{1}, \ldots, y_{n} \}$ with
$J_{\Delta} = I(\Delta^{\sharp})$.
It then follows that
a squarefree monomial $x_{i_{1}} \cdots x_{i_{s}}y_{j_{1}} \cdots y_{j_{t}}$
belongs to the minimal system of monomial generators of
$J_\Delta^\vee$ if and only if
$\{x_{i_{1}}, \ldots, x_{i_{s}}, y_{j_{1}}, \ldots, y_{j_{t}}\}$
is a minimal vertex cover (\cite[p.~156]{HH}) of $\Delta^{\sharp}$.

Since each facet of $\Delta^{\sharp}$ contains either $x_{i}$ or $y_{i}$
for $1 \leq i \leq n$, it follows that
$\{ x_{i}, y_{i} \}$ is a minimal vertex cover of $\Delta^{\sharp}$
for $1 \leq i \leq n$.

We claim $F = \{ y_{j} \, : \, j \in B \}$, where $B \subset [n]
= \{ 1, \ldots, n \}$,
is a vertex cover of $\Delta^{\sharp}$ if and only if
$F \not\in \Delta'$.
In fact, if $F \in \Delta'$,
then $\{ x_{j} \, : \, j \in B \} \cup \{ y_{i} \, : \, i \in [n] \setminus B \}$
is a facet of $\Delta^{\sharp}$.  Hence $F$ cannot be a vertex cover of
$\Delta^{\sharp}$.  Conversely suppose that $F \not\in \Delta'$.
Let $\{ x_{j} \, : \, j \in C \} \cup \{ y_{i} \, : \, i \in [n] \setminus C \}$
be a facet of $\Delta^{\sharp}$.  Then $\{ x_{j} \, : \, j \in C \}$ is a face
of $\Delta$.  Since $F \not\in \Delta'$, it follows that $B \not\subset C$.
Thus $B \cap ([n] \setminus C) \neq \emptyset$.  Hence $F$ is a vertex cover of
$\Delta^{\sharp}$.

Finally, suppose that
$G = \{ x_{i} \, : \, i \in A \} \cup \{ y_{j} \, : \, j \in B \}$
is a minimal vertex cover of $\Delta'$,
where $A \subset [n], B \subset [n]$ with $A \cap B = \emptyset$.
We claim $A = \emptyset$.  In fact, if $A \neq \emptyset$, then
$F = \{ y_{j} \, : \, j \in B \}$ cannot be a vertex cover of $\Delta^{\sharp}$.
Hence $F$ must be a face of $\Delta'$.  Then
$\{ x_{i} \, : \, i \in B \} \cup \{ y_{j} \, : \, j \in [n] \setminus B \}$
is a facet of $\Delta^{\sharp}$.  Since $A \cap B = \emptyset$,
it follows that $G$ cannot be a vertex cover of $\Delta'$.

In consequence, the minimal vertex covers of $\Delta^{\sharp}$ are either
$\{ x_{i}, y_{i} \}$ for $1 \leq i \leq n$ or the minimal nonfaces of $\Delta'$.
Since $I_{\Delta'}=I(\Gamma)$, the the minimal nonfaces of $\Delta'$ coincides with
the facets of $\Gamma$.  It then follows that $J_\Delta^\vee=I(W(\Gamma))$,
as desired.
\end{proof}

\begin{Corollary}
\label{always}
Let $\Delta$ be a simplicial complex. Then $J_\Delta$ has a linear resolution.
\end{Corollary}

\begin{proof}
By applying the Eagon-Reiner Theorem \cite{EG} (see also \cite[Theorem 8.1.9]{HH}) it suffices to show that $I(W(\Gamma))$ is a Cohen--Macaulay ideal. This is a well-known fact:  Notice that $I(W(\Gamma))$ is the polarization of $L=(I(\Gamma), y_1^2,\ldots,y_n^2)$. Since $\dim K[y_1,\ldots,y_n]/L=0$. it follows that $L$ is a Cohen-Macaulay ideal. Hence by \cite[Corollary 1.6.3]{HH}, $I(W(\Gamma))$ is a Cohen--Macaulay ideal as well.
\end{proof}

The next result describes the precise structure of the resolution of $S/J_\Delta$. For a simplicial complex $\Delta$ on $[n]=\{1,\ldots,n\}$ we let $\overline{\Delta}=\{[n]\setminus F\:\; F\in \Delta\}$.

\begin{Theorem}
\label{resolution}
Let $\Delta$ be a simplicial complex of dimension $d-1$ on the vertex set $[n]$. For each integer $j\geq 1$, let $F_j$ be the free $S$-module with basis elements $e_{G,H}$ indexed by  $G\in \Delta$ and $H\in \overline{\Delta}$ satisfying the condition that $|G\sect H|=j-1$ and $G\union H=[n]$. Furthermore, we set $F_0=S$. For each $j=2,\ldots, d$ we define the $S$-linear map
\[
\partial_j\:\; F_j\to F_{j-1}, \quad \text{with} \quad\partial_j(e_{G,H})=\sum_{i\in G\sect H}(-1)^{\sigma(G\sect H,i)}(x_ie_{G\setminus\{i\}, H}-y_ie_{G, H\setminus\{i\}}),
\]
where $\sigma(G\sect H,i)=|\{j\in G\sect H\:\; j<i\}|$.
Then
\[
\begin{CD}
\FF_\Delta\: 0 @>>> F_d @>\partial_d>> F_{d-1} @> \partial_{d-1} >>  \cdots @> \partial_2>> F_1@> \partial_1 >> F_0@>>> 0,
\end{CD}
\]
is the graded free resolution of $S/J_\Delta$, where $\partial_1(e_{G,H})=x_Gy_H$ for all $e_{G,H}\in F_1$.
\end{Theorem}

\begin{proof}
We first show that $\FF$ is a complex. One immediately verifies that $\partial_1\circ \partial_2=0$. Now let $j>2$ and $e_{G,H}\in F_j$. Set $L=G\sect H$. Then
\begin{eqnarray*}
&&\partial_j(\partial_{j-1}(e_{G,H}))=\partial_{j-1}(\sum_{i\in L}(-1)^{\sigma(L,i)}(x_ie_{G\setminus\{i\}, H}-y_ie_{G, H\setminus\{i\}}))\\
&=&\sum_{i\in L}(-1)^{\sigma(L,i)}[x_i(\sum_{k\in L\setminus\{i\}}(-1)^{\sigma(L\setminus\{i\},k)}(x_k e_{G\setminus\{i,k\},H}-y_ke_{G\setminus\{i\},H\setminus \{k\}}))\\
&-& y_i(\sum_{k\in L\setminus \{i\}}(-1)^{\sigma(L\setminus\{i\},k)}(x_k e_{G\setminus\{k\},H\setminus\{i\}}-y_ke_{G,H\setminus \{i,k\}}))]\\
&=& \sum_{i,k\in L,i<k}\sigma_{ik} (x_ix_ke_{G\setminus\{i,k\},H}-x_iy_ke_{G_\setminus\{i\},H\setminus \{k\}} +y_iy_ke_{G,H\setminus\{i,k\}}).
\end{eqnarray*}
where $\sigma_{ik}=(-1)^{\sigma(L,i)+\sigma(L\setminus\{i\},k)}+(-1)^{\sigma(L,k)+\sigma(L\setminus\{k\},i)}$. Since $\sigma_{ik}=0$ for all $i<k$ it follows that $\partial_j(\partial_{j-1}(e_{G,H}))=0$, as desired.

We will prove the acyclicity of $\FF$ be induction on $|\Delta|$.  The assertion is trivial if $\Delta=\{\emptyset, \{1\}\}$. Now let $n>1$, and let $F$ be a facet of $\Delta$. We set   $\Gamma=\Delta\setminus\{F\}$. We may assume that $\Gamma$ is a simplicial complex on the vertex set $[n']$ where $n'\leq n$.  By induction hypothesis,  $\FF_\Gamma$ is a graded minimal free $S'$-resolution of $S'/J_\Gamma$, where $S'=K[x_1,\ldots,x_{n'},y_1,\ldots,y_{n'}]$. Thus $\GG=\FF_\Gamma\tensor_{S'} S$ is a graded minimal free $S$-resolution of $S/J_\Gamma S$, and we obtain an exact sequence of complexes
\begin{eqnarray}
\label{exact}
\begin{CD}
0 @>>>\GG @>\varphi >> \FF_\Delta @>\epsilon>> \KK\to 0.
\end{CD}
\end{eqnarray}
For all $H\subset [n']$ we set $H'=H\union \{n'+1,\ldots,n\}$. Then $\varphi\: \GG\to \FF_\Delta$ is defined by  $\varphi(e_{G,H})=e_{G,H'}$  for all $e_{G,H}\in \GG$.   Furthermore,  $\KK=F_\Delta/\varphi(\GG)$.

One verifies that $K_j$ is a free module admitting the  basis $\bar{e }_{F,H}=\epsilon(e_{F,H})$ with  $H\subset [n]$ such that $|F\sect H|=j-1$ and $F\union H=[n]$. Denote by $\delta$ the differential of $\KK$. Then
\[
\delta_j(\bar{e }_{F,H})=-\sum_{i\in F\sect H}(-1)^{\sigma(F\sect H,i)}y_i\bar{e}_{F,H\setminus\{i\}}.
\]
Thus we see that $\KK$ is isomorphic to the Koszul complex attached to  the sequence $(y_i)_{i\in F}$,  homologically shifted by $1$. In particular, it follows that $H_j(\KK)=0$ for $j\neq 1$, while $H_1(\KK)=S/(y_i: i\in F)$.

Thus from the long exact sequence attached to (\ref{exact}) we obtain that $H_j(\FF_\Delta)=0$ for $j>1$ by using that $\GG$ is acyclic by our induction hypothesis. Furthermore, we obtain the exact sequence
\[
0\To H_1(\FF_C)\To H_1(\KK)\To H_0(\GG) \To H_0(\FF_C)\To 0.
\]
Since $H_0(\GG)= S/J_\Gamma S$ and $H_0(\FF_\Delta)=S/J_{\Delta}$ we see that $\Ker (H_0(\GG) \To H_0(\FF_C))=J_\Delta/J_\Gamma S$. Now $J_\Delta/J_\Gamma S\iso S/(J_\Gamma S: x_Fy_{[n]\setminus F})=S/(y_i: i\in F)=H_1(\KK)$. This proves that $H_1(\FF_C)=0$ and completes the proof of the theorem.
\end{proof}

\begin{Corollary}
\label{betti}
Let $\Delta$ be a simplicial complex of dimension $d-1$ with  $f$-vector $(f_{-1},f_0,\ldots,f_{d-1})$. Then
\[
\beta_j(J_\Delta)=\sum_{i=-1}^{d-1}{i+1\choose j}f_i.
\]
In particular, $\projdim J_\Delta=\dim \Delta+1$.
\end{Corollary}

\section{Whisker complexes and shellability}

 In this section we show that the face ideal of a simplicial complex does not only have  a linear resolution but  even linear quotients.

 \begin{Theorem}
 \label{linearquotients}
 Let $\Delta$ be a simplicial complex. Then $J_\Delta$ has linear quotients.
 \end{Theorem}

 \begin{proof}
 We choose any total order of the generators of $J_\Delta$ with the property that $u_G>u_F$ if $G\subset F$, and claim that $J_\Delta$ has linear quotients with respect to this order of the generators.

 Indeed, let $F\in \MF(\Delta)$,  and let $J(F)$ be the ideal generated by all $G$ with $u_G>u_F$. For any $G\in \MF(\Delta)$ one has
 \[
 (u_G):u_F=\frac{u_G}{\gcd(u_G,u_F)}=x_{G\setminus F}y_{F\setminus G}.
 \]
Now let $u_G\in J(F)$. Then $F\setminus G\neq \emptyset$. Let  $x_j\in F\setminus G$, and  let $H=F\setminus \{x_j\}$. Then $u_H>u_F$ and $(u_H):u_F=(y_j)$. Since $y_j$ divides  $y_{F\setminus G}$ it follows that $y_j$ divides $(u_G):u_F$, as desired.
 \end{proof}

\begin{Corollary}
 \label{shellable}
Let $\Gamma$ be a simplicial complex. Then the independence complex of  $W(\Gamma)$  is shellable.
\end{Corollary}

Theorem~\ref{linearquotients} can be generalized as follows.

\begin{Theorem}
\label{general}
Let $\MS$ be a non-empty  collection of subsets of $\{x_1,\ldots,x_n\}$ satisfying the following conditions:
\begin{enumerate}
\item[{\em (i)}] if $F,G\in \MS$, then $F\sect G\in \MS$;

\item[{\em (ii)}] for $F, G\in \MS$ with $G\subset F$  there exists $x_i\in F\setminus G$ such that $ F\setminus\{x_i\}\in \MS$.
\end{enumerate}
Then $I_\MS$ has linear quotients.
\end{Theorem}

\begin{proof}
 We choose any total order of the generators of $J_\Delta$ with the property that $u_G>u_F$ if $G\subset F$, and claim that $J_\MS$ has linear quotients with respect to this order of the generators. Let $u_G>u_F$,  and let $\MN=\{i\:\; F\setminus \{x_i\}\in \MS\}$. Then, by (ii), $\MN\neq\emptyset$. We claim that
 \[
 (u_G\:\; u_G>u_F)=(y_i\:\; i\in\MN).
 \]
To see why this is true, let $u_G>u_F$. Then, by (i), $H=G\sect F$ belongs to $\MS$, and by (ii) there exists $i\in\MN$ such that $x_i\in F\setminus H$.  Then $x_i\not G$,  and hence $y_i$ divides $(u_G):u_F$. Since $(u_{F_\setminus\{x_i\}}):u_F=(y_i)$, we are done.
\end{proof}

\section{Chain ideals and antichain ideals}

In this section we consider special classes of face ideals arising from finite partially ordered sets (posets, for short).
Let $P=\{p_1,\ldots,p_n\}$ be a finite poset, and $S=K[x_1,\ldots, x_n,y_1,\ldots,y_n]$ the polynomial ring in $2n$ variables over the field $K$. For each chain $\alpha$ of $P$ we define the monomial $u_\alpha\in S$ by setting
\[
u_\alpha=(\prod_{p_i\in \alpha}x_i)(\prod_{p_j\not\in \alpha}y_j),
\]
and let $I_C(P)$ be the ideal generated by the monomials $u_\alpha$ where $\alpha$ is a chain of $P$. We call $I_C(P)$ the \emph{chain ideal} of $P$.

Similarly, for each antichain $\beta$ of $P$ we define the monomial $u_\beta\in S$ by setting
\[
u_\beta=(\prod_{p_i\in \beta}x_i)(\prod_{p_j\not\in \beta}y_j),
\]
and let $I_A(P)$ be the ideal generated by the monomials $u_\beta$ where $\beta$ is an antichain of $P$. We call $I_A(P)$ the \emph{antichain ideal} of $P$.

Recall that the {\em comparability graph} of $P$ is a finite simple graph
${\mathcal C}(P)$ on $[n]$ whose edges are those subsets
$\{ i, j \}$ such that $p_{i}$ and $p_{j}$ are comparable in $P$, i.e.,
$\{ p_{i}, p_{j} \}$ is a chain of $P$.
Similarly, the {\em incomparability graph} of $P$ is a finite simple graph
${\mathcal A}(P)$ on $[n]$ whose edges are those subsets
$\{ i, j \}$ such that $p_{i}$ and $p_{j}$ are incomparable in $P$, i.e.,
$\{ p_{i}, p_{j} \}$ is an antichain of $P$.

\begin{Theorem}
\label{chain}
Let $P$ be a finite poset.

{\em (a)} The Alexander dual of the chain ideal $I_C(P)$ is the edge ideal of the whisker graph of the incomparability graph of $P$.

{\em (b)} The Alexander dual of the antichain ideal $I_A(P)$ is the edge ideal of the whisker graph of the comparability graph of $P$.
\end{Theorem}

\begin{proof}
(a) Let $\Delta$ be a simplicial complex on $\{y_{1}, \ldots, y_{n}\}$
whose faces are those $F \subset \{y_{1}, \ldots, y_{n}\}$
with $\{ p_{i} \, : \, y_{i} \in F \}$ is a chain of $P$.
Theorem \ref{faceideal} says that the Alexander dual $I_C(P)^{\vee}$ of
$I_C(P)$ coincides with $I(W(\Gamma))$, where $\Gamma$ is a simplicial complex on
$\{y_{1}, \ldots, y_{n}\}$ with $I_{\Delta} = I(\Gamma)$.
A subset $F \subset \{y_{1}, \ldots, y_{n}\}$ is nonface of $\Delta$
if and only if an antichain of $P$ is contained in $F$.
Thus the minimal nonfaces of
$\Delta$, which coincides with the facets of $\Gamma$ are those subset
$\{y_{i}, y_{j}\}$ such that $\{ p_{i}, p_{j} \}$ is an antichain of $P$.
Thus $I(\Gamma)$ is the edge ideal of the incomparagraph of $P$.
Hence $I(W(\Gamma))$ is the edge ideal of the whisker graph
of the incomparability graph of $P$, as desired.

(b) Let $\Delta$ be a simplicial complex on $\{y_{1}, \ldots, y_{n}\}$
whose faces are those $F \subset \{y_{1}, \ldots, y_{n}\}$
with $\{ p_{i} \, : \, y_{i} \in F \}$ is an antichain of $P$.
Theorem \ref{faceideal} says that the Alexander dual $I_A(P)^{\vee}$ of
$I_A(P)$ coincides with $I(W(\Gamma))$, where $\Gamma$ is a simplicial complex on
$\{y_{1}, \ldots, y_{n}\}$ with $I_{\Delta} = I(\Gamma)$.
A subset $F \subset \{y_{1}, \ldots, y_{n}\}$ is nonface of $\Delta$
if and only if a chain of $P$ is contained in $F$.
Thus the minimal nonfaces of
$\Delta$, which coincides with the facets of $\Gamma$ are those subset
$\{y_{i}, y_{j}\}$ such that $\{ p_{i}, p_{j} \}$ is a chain of $P$.
Thus $I(\Gamma)$ is the edge ideal of the comparagraph of $P$.
Hence $I(W(\Gamma))$ is the edge ideal of the whisker graph
of the comparability graph of $P$, as desired.
\end{proof}

The {\em Dilworth number} of a finite poset $P$ is the least number of chains
into which $P$ can be partitioned.  Dilworth's theorem \cite{D}
guarantees that the Dilworth number of $P$ is equal to the maximal cardinality
of the antichains of $P$.
%
% \bibitem{Dilworth}
% R.~P.~Dilworth, A Decomposition Theorem for Partially Ordered Sets,
% Annals of Math. {\bf 51} (1950), 161--166.
%
\begin{Corollary}
The chain ideal and the antichain ideal of a finite poset have a linear resolution. Moreover, $\projdim I_C(P)=\rank P+1$  while $\projdim I_A(P)$ is the Dilworth number of $P$.
\end{Corollary}

\section{Higher dimensional whiskers}

The purpose of this section is to generalize Corollary~\ref{shellable}. Let $\Delta$ be a simplicial complex on $\{x_1,\ldots,x_n\}$. Given positive integers  $k_1,\ldots,k_n$ and $d_1,\ldots,d_n$ with $d_i\leq k_i$ for all $i$, we define  the {\em higher dimensional whisker complex} $W^{d_1,\ldots,d_n}_{k_1,\ldots,k_n}(\Delta)$ of $\Delta$ to be  the simplicial complex on the vertex set
\[
x_1, x_1^{(1)},\cdots,x_1^{(k_1)},x_2,x_2^{(1)},\cdots,x_2^{(k_2)},\cdots,x_n,x_n^{(1)},\cdots,x_n^{(k_n)},
\]
whose facets are the facets of $\Delta$ together with all subsets of cardinality $d_i+1$ of $\{x_i, x_i^{(1)},\ldots,x_i^{(k_i)}\}$  for  $i=1,\ldots,n$.  These subsets are called the whiskers of $\Delta$.

Note that the wisker complex of $\Delta$ as defined in Section 1 is just the complex  $W^{1,\ldots,1}_{1,\ldots,1}(\Delta)$. See Figure~1 for an example of a higher whisker complex.

\begin{figure}[hbt]
\begin{center}
\psset{unit=1cm}
\begin{pspicture}(-6,-4)(5,5)

\rput(0,0){$\bullet$}
\rput(1.5,1){$\bullet$}
\rput(0,2){$\bullet$}

%\psline(0,0)(0,2)
%\psline(1.5,1)(0,2)
%\psline(1.5,1)(0,0)
\pspolygon(0,0)(0,2)(1.5,1)(0,0)
\psline(1.5,1)(3,1)
\rput(3,1){$\bullet$}
\rput(3.5,1.1){$x_3^{(1)}$}

\pspolygon[style=fyp,fillcolor=light](0,0)(0,2)(-1.5,1)(0,0)
\rput(-1.5,1){$\bullet$}

\pspolygon[style=fyp,fillcolor=light](-1.5,1)(-4,-0.5)(-4,2.5)
\psline(-1.5,1)(-3,1)
\psline(-3,1)(-4,2.5)
\psline(-3,1)(-4,-0.5)

\rput(-3,1){$\bullet$}
\rput(-4,2.5){$\bullet$}
\rput(-4,-0.5){$\bullet$}

\pspolygon(0,0)(-1.5,-2.5)(1.5,-2.5)
\psline(0,0)(0,-1.5)
\psline(0,-1.5)(-1.5,-2.5)
\psline(0,-1.5)(1.5,-2.5)

\rput(0,-1.5){$\bullet$}
\rput(-1.5,-2.5){$\bullet$}
\rput(1.5,-2.5){$\bullet$}

\pspolygon(0,2)(-1, 3.5)(1,3.5)
\rput(-1,3.5){$\bullet$}
\rput(1,3.5){$\bullet$}

\rput(-1.5, 0.7){$x_1$}
\rput(0.4,0){$x_2$}
\rput(1.5,0.7){$x_3$}
\rput(0.4,2){$x_4$}

\rput(-3.8, -0.75){$x_1^{(2)}$}
\rput(-3.8,2.9){$x_1^{(1)}$}
\rput(-3.45,1){$x_1^{(3)}$}
\rput(-0.9,3.9){$x_4^{(1)}$}
\rput(1.2,3.9){$x_4^{(2)}$}

\rput(-1.7,-2.9){$x_2^{(1)}$}
\rput(1.8,-2.9){$x_2^{(2)}$}
\rput(0.1,-2){$x_2^{(3)}$}

\end{pspicture}
\end{center}
\caption{}\label{butterfly}
\end{figure}

\begin{Theorem}
\label{generalized}
The independence complex of a higher whisker complex is shellable.
\end{Theorem}

\begin{proof}
Let $W^{d_1,\ldots,d_n}_{k_1,\ldots,k_n}(\Gamma)$ be the whisker complex,  $\Delta$ its independence complex and $I=I_{\Delta^\vee}$.
We will show that $I$ has linear quotients. This is equivalent to say that $\Delta$ is shellable. Note that the generators of $I$ correspond bijectively to the vertex covers of $W^{d_1,\ldots,d_n}_{k_1,\ldots,k_n}(\Gamma)$. Indeed, if $$C\subset \{x_1, x_1^{(1)},\ldots,x_1^{(k_1)},x_2,x_2^{(1)},\ldots,x_2^{(k_2)},\ldots,x_n,x_n^{(1)},\ldots,x_n^{(k_n)}\}$$ is a minimal vertex cover of $W^{d_1,\ldots,d_n}_{k_1,\ldots,k_n}(\Gamma)$, then the product of the elements in $C$ is the corresponding generator of $I$.

We claim, that $C$ is a minimal vertex cover of $W^{d_1,\ldots,d_n}_{k_1,\ldots,k_n}(\Gamma)$ if and only if the following conditions are satisfied:
\begin{enumerate}
\item[(1)] $C\sect \{x_1,\ldots,x_n\}$ is a vertex cover of $\Gamma$;
\item[(2)] $C\sect \{x_i, x_{1}^{(i)}, \ldots,x_{k_i}^{(i)}\}$ is a minimal vertex cover of the $d_i$-skeleton of the simplex on $\{x_i, x_{i1}, \ldots,x_{ik_i}\}$.
\end{enumerate}
Clearly any set $C$ of vertices satisfying (1) and (2) is a minimal vertex cover.

We denote by  $u_C\in I$ the  monomial corresponding to the vertex cover $C$ of  $W^{d_1,\ldots,d_n}_{k_1,\ldots,k_n}(\Gamma)$. Then
\[
u_C=x_{C_0}x_{C_1}\cdots x_{C_n},
\]
where
\[
x_{C_0}=\prod_{x_j\in C}x_i \quad  \text{and}\quad  x_{C_j}=\prod_{\{i\:\; x_j^{(i)}\in C\}}x_j^{(i)}
\]
for $j=1,\ldots,n$.

We now define a total order of the monomial generators of $I$ as follows: let $u_C=x_{C_0}x_{C_1}\cdots x_{C_n}$  and $u_D=x_{D_0}x_{D_1}\cdots x_{D_n}$. Then $u_C>u_D$ if either $x_{C_0}>x_{D_0}$ with respect to the degree  lexicographic order, or $x_{C_0}=x_{D_0}$ and $x_{C_1}\cdots x_{C_n}>x_{D_1}\cdots x_{D_n}$ with respect to the lexicographic order induced by
\[
x_1^{(1)}>\ldots >x_1^{(k_1)}>x_2^{(1)}>\ldots >x_2^{(k_2)}>\ldots >x_n^{(1)}>\ldots>x_n^{(k_n)}.
\]
We claim  that with this ordering of the generators, $I$ has linear quotients. We first note that all generators of $I$ have the same degree, namely, $\sum_{i=1}^n(k_i-d_i)+n$. We must show that for all $u_C$ the colon ideal $(u_D\: u_D>u_C):u_C$ is generated by variables. Let   $u_D>u_C$ with $x_{D_0}\neq x_{C_0}$. Then $x_{D_0}>x_{C_0}$, and hence there exists $x_j$ such that $x_j$ divides $x_{D_0}$ but does not divide $x_{C_0}$. Let $u_E$ be the generator with $x_{E_0}=x_{C_0}x_j$,  $x_{E_i}=x_{C_i}$ for $i\neq j$ and $x_{E_j}=x_{C_j}/x_j^{(l)}$ where $x_j^{(l)}$ divides $x_{C_j}$. Then $u_E>u_C$ and $(u_E):u_C=(x_j)$. Since $x_j$ divides $(u_D):u_C$, and we are done in this case.

We now consider the case $x_{D_0}=x_{C_0}$. Let $\MA=\{u_E\:\ x_{E_0}=x_{C_0}\}$. For each $i\geq1$, let    $\MA_i=\{x_{E_i}\:\ u_E\in \MA\}$. If $x_i$ divides $x_{C_0}$,  then  $\MA_i$ is the set of all monomials of degree $k_i-d_i$  in the variables $x_{1}^{(i)}, \ldots,x_{k_i}^{(i)}$, and if $x_i$ does not divide $x_{C_0}$,  then  $\MA_i$ is the set of all monomials of degree $k_i-d_i+1$ in the same set of variables. Note that $\MA_i$ generates a matroidal ideal. Moreover the product of matroidal ideals in pairwise disjoint sets of variables is again a matroidal ideal. It is known, see \cite[Theorem 12.6.2]{HH}, that matroidal ideals have linear quotients with respect to the lexicographic order of the generators. This completes the proof of the theorem.
\end{proof}

\end{document}